\documentclass[11pt, reqno]{amsart} 
\usepackage{amsfonts,amssymb,amscd,amsthm,amsbsy,bbm,epsf,calc}
\usepackage{hyperref}
\hypersetup{
    colorlinks = true,
	linkcolor = {blue},
	citecolor = {blue},
}
\usepackage{dsfont} 
\usepackage{color, float}
\usepackage{datetime}
\usepackage{bm}
\usepackage[pdftex]{graphicx}
	
\makeatletter
\def\l@section{\@tocline{1}{0pt}{1pc}{}{}}
\def\l@subsection{\@tocline{2}{0pt}{1pc}{4.6em}{}}
\def\l@subsubsection{\@tocline{3}{0pt}{1pc}{7.6em}{}}
\renewcommand{\tocsection}[3]{%
  \indentlabel{\@ifnotempty{#2}{\makebox[2.3em][l]{%
    \ignorespaces#1 #2.\hfill}}}#3}
\renewcommand{\tocsubsection}[3]{%
  \indentlabel{\@ifnotempty{#2}{\hspace*{2.3em}\makebox[2.3em][l]{%
    \ignorespaces#1 #2.\hfill}}}#3}
\renewcommand{\tocsubsubsection}[3]{%
  \indentlabel{\@ifnotempty{#2}{\hspace*{4.6em}\makebox[3em][l]{%
    \ignorespaces#1 #2.\hfill}}}#3}
\makeatother

\numberwithin{equation}{section}

\newcounter{hours}\newcounter{minutes}

\theoremstyle{definition}
\newtheorem{theorem}{Theorem}[section]

\newtheorem{lemma}[theorem]{Lemma}

\newtheorem{proposition}[theorem]{Proposition}

\theoremstyle{remark}                  
\newtheorem{remark}[theorem]{Remark}

\theoremstyle{definition}
\newtheorem{definition}[theorem]{Definition}

\newtheorem{problem}[theorem]{Problem}

\def\diam{\textnormal{diam}}

\def\det{\textnormal{det}}

\def\Id{\textnormal{Id}}

\newcommand{\norm}[1]{\lvert#1\rvert}

\newcommand{\inner}[2]{\langle #1,#2\rangle}

\DeclareMathOperator{\dom}{dom}

\DeclareMathOperator{\spn}{span}

\def\zbar{\bar{z}}
\def\qbar{\bar{q}}

\def\Nbar{\bar{N}}

\newcommand{\cExp}[2]{exp^c_{#1}({#2})}
\newcommand{\cstarExp}[2]{exp^{c^*}_{#1}({#2})}

\newcommand\normal[2]{N_{#1}\left(#2\right)}

\newcommand{\paren}[1]{\left(#1\right)}

\def\ybar{\bar{y}}
\def\vbar{\bar{v}}

\newcommand{\Hvol}[2][]{\mathcal{H}^{#1}\paren{#2}}

\newcommand{\subdiff}[2]{\partial {#1}{(#2)}}

\def\S{\mathbb{S}}

\DeclareMathOperator{\proj}{proj}
\DeclareMathOperator{\graph}{graph}
\newcommand{\dVol}[1][]{d \textnormal{Vol}_{#1}}
\newcommand{\R}{\mathbb{R}}
\DeclareMathOperator{\epi}{epi}
\def\Ybar{\bar{Y}}
\begin{document}
\title{Optimal transport and the Gauss curvature equation}

\author{Nestor Guillen}
\address[Nestor Guillen]{Department of Mathematics, Texas State University}
\email{nestor@txstate.edu} 

\author{Jun Kitagawa}
\address[Jun Kitagawa]{Department of Mathematics, Michigan State University}
\thanks{}
\email{kitagawa@math.msu.edu}
\thanks{JK's research was supported in part by National Science Foundation grant DMS-1700094. NG's research was supported in part by National Science Foundation Grant DMS-1700307.}
\begin{abstract}
In this short note, we consider the problem of prescribing the Gauss curvature and image of the Gauss map for the graph of a function over a domain in Euclidean space. The prescription of the image of the Gauss map turns this into a second boundary value problem. Our main observation is that this problem can be posed as an optimal transport problem where the target is a subset of the lower hemisphere of $\mathbb{S}^n$. As a result we obtain existence and regularity of solutions under mild assumptions on the curvature, as well as a quantitative version of a gradient blowup result due to Urbas, which turns out to fall within the optimal transport framework.
\end{abstract}

\date{\today}

\date{}
\maketitle
\markboth{N. Guillen and J. Kitagawa}{Optimal transport and the Gauss curvature equation}


\begin{center}
  \emph{To John Urbas with great admiration, on the occasion of his 60th birthday.}
\end{center}

\section{Introduction}

The aim of this note is not to present substantially new results but to show how the theory of optimal transport, as developed since Yann Brenier's fundamental work \cite{Bre91}, provides a simple framework to construct solutions to the equation of prescribed Gauss curvature where we consider the second boundary value problem, as opposed to the more commonly studied Dirichlet problem. Furthermore, we provide a quantitative proof and interpretation in terms of optimal transport of a result due to Urbas in \cite{Urbas84} concerning the case where the total mass of the function $K(x)$ is critical with respect to admissible curvatures (see \eqref{eqn:critical mass condition}). 

We recall that when $u:\Omega \to \mathbb{R}$ is a smooth function in a smooth domain $\Omega \subset \mathbb{R}^n$ then the Gauss curvature of its graph ${\Sigma_u} = \{ (x,u(x)) \in \mathbb{R}^{n+1} \mid x \in \Omega \}$ at the point $(x,u(x))$ is given by
\begin{align}\label{eqn:Gauss curvature in terms of u}
  K(x) = \frac{\det(D^2u(x))}{(1+|\nabla u(x)|^2)^{\frac{n+2}{2}}}.
\end{align}
The question we are concerned with is the construction of a surface with given $K(x)$ and a given image for its Gauss map, thus we pose the following problem.  
\begin{problem}\label{problem:main}
  Given $\Omega \subset \mathbb{R}^n$ convex and bounded, a non-negative function $K:\Omega\to\mathbb{R}$, and a domain $\Omega^* \subset \mathbb{S}^n$, find a convex function $u:\Omega \to \mathbb{R}$ whose graph ${\Sigma_u}$ is such that the Gauss curvature at $(x,u(x))$ is equal to $K(x)$ and such that the image of ${\Sigma_u}$ under the Gauss map (outer unit normal map) is equal to $\Omega^*$.

\end{problem}
The prescribed Gauss curvature problem is traditionally formulated as a boundary value problem for the PDE \eqref{eqn:Gauss curvature in terms of u} with prescribed values for $u$ along $\partial \Omega$ (Dirichlet problem), or to a lesser extent the normal derivative (Neumann problem). Prescribing the image of the Gauss map as in Problem \ref{problem:main} is equivalent to what is known in the elliptic PDE literature as a second boundary value problem.

The main observation in this paper is that Problem \ref{problem:main} is an optimal transport problem between $\Omega$ and a hemisphere of $\mathbb{S}^n$ (Theorem \ref{theorem:existence and uniqueness}). In this transport problem the source measure is given by the Gauss curvature $K(x)$ restricted to $\Omega$ and the target measure is the $n$-form $-y_{n+1}\dVol[\mathbb{S}^n]$ restricted to the corresponding target set in $\mathbb{S}^n$.

This identification of Problem \ref{problem:main} as an optimal transport problem immediately puts at our disposal an array of general results and methods from optimal transport. Accordingly, we have existence and uniqueness of weak solutions (also Theorem \ref{theorem:existence and uniqueness}), as well as regularity of such weak solutions under further assumptions on $K$ (Theorem \ref{theorem:subcritical case regularity}). 

The optimal transport framework gives another point of view on a result of Urbas (\cite{Urbas84}) concerning prescribed Gauss curvature ``without boundary conditions''. In particular, we demonstrate the use of optimal transport techniques to obtain a quantitative version of \cite[Lemma 6.3]{Urbas84}. This result says that if the mass of $K$ is the maximum possible admissible value, and $K$ satisfies a growth condition near $\partial \Omega$, then the slope of $u(x)$ must go to infinity as $x$ approaches the boundary. In this paper, we provide an explicit rate for the blow up.

\subsection*{Outline of the paper} The rest of this paper is organized as follows: in Section \ref{section:results} we go over basic terminology, state the main results, and discuss some relevant literature. Then in Section \ref{section: OT} we review needed definitions and results from optimal transport. In Section \ref{section:Gauss Curvature as OT} we show how to formulate Problem \ref{problem:main} as an optimal transport problem and from here prove existence of solutions and some regularity results. Lastly, in Section \ref{section:Urbas revisited} we revisit Urbas' result.

\section{Main results}\label{section:results}

Let $\Omega\subset \R^n$ be bounded, let $K:\Omega \to \mathbb{R}$ be a Borel, non-negative function such that
\begin{align}\label{eqn:K mass condition}
  \int_{\Omega} K(x)\;dx \leq \omega_n,
\end{align}
where $\omega_n$ is the Lebesgue measure of the unit ball in $\R^n$, and suppose $\overline{\Omega}=\overline{\{x\in \Omega\mid K(x)>0\}}$. Also if $u$ is a convex function defined on $\Omega$ and $\dom(Du)$ the set of differentiability points of $u$ in $\Omega$, for each $x\in \dom(Du)$ we define $\mathcal{G}u(x)$ to be the unique vector in $\S^n_-$ which is the outward pointing normal vector to the graph of $u$ at $(x, u(x))$. Of course if $u\in C^1(\Omega)$, this is the classical Gauss map parametrized by points in $\Omega$.

We refer the reader to Section \ref{section: OT} below for definitions specific to optimal transport theory. 
\begin{theorem}\label{theorem:existence and uniqueness}
  Let $K$ be as above and let $\Omega^* \subseteq \mathbb{S}^n_-$ be such that 
  \begin{align*}
    \int_{\Omega}K(x)\;dx = \int_{\Omega^*}(-y_{n+1})\dVol[\mathbb{S}^n].
  \end{align*}
  
  Then any Brenier solution $u$ of the optimal transport problem with cost function $c: \Omega\times \S^n_-\to \R$, $c(x, \ybar):=\inner{x}{\frac{y}{y_{n+1}}}$, between the measures $K dx$ on $\Omega$ and $-y_{n+1}\mathds{1}_{\Omega^*}\dVol[\mathbb{S}^n]$ solves the equation \eqref{eqn:Gauss curvature in terms of u} pointwise a.e. in $\Omega$, and
  \begin{align}\label{eqn: gauss map image}
      \overline{\mathcal{G}u(\Omega)}=\overline{\Omega^*}.
  \end{align}
  
If $u\in C^2(\Omega)$ then $u$ is a classical solution of \eqref{problem:main}, and the image of the Gauss map of ${\Sigma_u}$ is equal to $\Omega^*$.
\end{theorem}

Next, assume further that $K(x)$ satisfies the following bounds for two constants $\lambda_1$ and $\lambda_2$.
\begin{align}\label{eqn:K upper and lower bounds}
  0<\lambda_1 \leq K(x) \leq \lambda_2, \quad\forall\;x\in\Omega.  
\end{align}

\begin{theorem}\label{theorem:subcritical case regularity}
In the previous theorem suppose furthermore that $K(x)$ satisfies \eqref{eqn:K upper and lower bounds} and $\Omega^*$ is a geodesically convex subset of $\S^n_-$.

If $\Omega^*$ is compactly contained in $\S^n_-$, then $u$ is strictly convex and $u \in C^{1,\alpha}_{\textnormal{loc}}(\Omega)$ for some $\alpha\in (0, 1]$.  If $\overline{\Omega^*}\cap \partial \S^n_-\neq \emptyset$ and $\Omega$ is uniformly convex, then for any $\Omega'\Subset\Omega$, there exists an $\alpha\in (0, 1]$ such that $u\in C^{1, \alpha}_{\textnormal{loc}}(\Omega')$.
\end{theorem}
\begin{remark}
 We remark here, both the conditions and conclusions we have stated in Theorems \ref{theorem:existence and uniqueness} and \ref{theorem:subcritical case regularity} are by no means the sharpest or the strongest conclusions that are currently known. For example, it is possible to obtain global or higher regularity results under various assumptions, which we do not detail here.
\end{remark}

Finally, we show a quantitative version of the result \cite[Lemma 6.3]{Urbas84} by Urbas, using optimal transport techniques.
\begin{theorem}\label{theorem:critical case boundary behavior}
  Let $u$ be a Brenier solution as in the previous theorem. Suppose furthermore that equality is attained in \eqref{eqn:K mass condition}, there exist $C_0$, $r_0>0$ and $\delta\in (0, 1)$ such that 
  \begin{align}\label{eqn: K decay condition}
      K(x)\leq \frac{C_0}{d(x, \partial \Omega)^\delta},\ \forall x\in \mathcal{N}_{r_0}(\partial \Omega),
  \end{align} 
  and there exists $R_0>0$ such that $\Omega$ satisfies an enclosing ball condition at every point on $\partial \Omega$ for some ball of radius less than or equal to $R_0$.
  
  If $x_0\in \Omega$ and $p_0\in \subdiff{u}{x_0}$, are such that
  \begin{align}
      d(x_0, \partial \Omega)&\leq \min(\frac{\rho^2}{16(1+4R_0)},  \frac{r_0}{2}, \frac{1}{4}),\label{eqn: d_0 condition}
  \end{align}
 (where $\rho$ is from Remark \ref{rem: Omega boundary}) then 
  \begin{align}\label{eqn: p_0 estimate}
    |p_0| \geq \Lambda d(x,\partial \Omega)^{-\frac{1-\delta}{n+2} }-2,\; \Lambda = \Lambda(R_0,L,n,\delta).
  \end{align}
  (where $L>0$ is also from Remark \ref{rem: Omega boundary}).
\end{theorem}

\subsection{Relation with previous results}

The results in this note will not be surprising to the reader familiar with the Gauss curvature equation or with the optimal transport problem. That being said, we are not aware of any past work that explicitly mentions Problem \ref{problem:main} in the form stated there. In the past, the prescribed Gauss curvature problem has been viewed as an optimal transport problem where the target domain lies in Euclidean space \cite{McCann95,Urbas84}, rather than the hemisphere as in Problem \ref{problem:main}. Moreover, the prescribed Gauss curvature equation has generally been studied with a Dirichlet boundary condition, which is not as amenable to optimal transport methods; the standard ``boundary condition'' arising in optimal transport theory is to prescribe the image of the transport map, which corresponds to a (nonlinear) second boundary value condition from PDE theory. The motivation behind the formulation in Problem \ref{problem:main} is that one natural choice of ``boundary condition'' different from Dirichlet conditions for the prescribed Gauss curvature problem would be to prescribe the image of the Gauss map directly.

In \cite{Urbas84}, Urbas studied solutions to the prescribed Gauss curvature equation where $K$ has the maximal allowable mass, that is
\begin{align}\label{eqn:critical mass condition}
  \int_{\Omega}K(x)\;dx = \omega_n,    
\end{align}
and showed that a solution to this problem are uniquely determined without specifying any Dirichlet boundary conditions. Moreover, Urbas showed the gradient of the solution blows up along the boundary of $\Omega$, under a growth condition on $K$ near $\partial \Omega$; a similar result can be expected in the optimal transport framework. First, a solution to the Gauss curvature equation \eqref{eqn:Gauss curvature in terms of u} where $K$ is non-negative and satisfies \eqref{eqn:critical mass condition} must automatically satisfy $\mathcal{G}u(\Omega) = \mathbb{S}^n_-$ by mass balance --in particular, this can be viewed as a ``hidden'' boundary condition for this case. Secondly, under the optimal transport formulation one expects the boundary of $\Omega$ to be mapped by the transport map to the boundary of $\Omega^* = \mathbb{S}^n_-$, and thus the normal to the graph of $u$ should become horizontal as one approaches $\partial \Omega$. We show a quantitative version of this claim in Theorem \ref{theorem:critical case boundary behavior}.

One can also obtain regularity results via the optimal transport theory. The formulation as an optimal transport problem leads to regularity results in a more or less straightforward manner.  The key geometric condition in order to invoke the optimal transport regularity theory turns out to be equivalent to geodesic convexity of the prescribed image of the Gauss map. If the target domain does not touch the boundary of the hemisphere, the standard optimal transport theory in \cite{GK14} can be applied, while when the target domain reaches the boundary of the hemisphere we can appeal to the non-compact case theory in \cite{CorderoFigalli19}.

We mentioned the equivalence between Problem \ref{problem:main} and a second boundary value problem, in this regard we refer to previous work by Delano\"e for the two dimensional case \cite{Delanoe91} and by Urbas for general dimensions \cite{Urbas97}. We also mention work by Nadirashvilli \cite{Nadirashvili83} on uniqueness for the second boundary value problem, including equations for prescription of curvature. 

Finally, we raise one possible future direction. Suppose $\Omega$ is a $C^{1, 1}$, uniformly convex domain. Fix a curvature function $K\in C^{1, 1}(\Omega)\cap C^{1, 1}(\overline\Omega)$, $K \geq 0$ whose total mass is strictly smaller than $\omega_n$, satisfying $K(x)\leq C d(x, \partial \Omega)$ for some $C>0$ in some neighborhood of $\partial \Omega$. Then for any $\phi\in C^{1, 1}(\overline\Omega)$, it is possible to find a solution $u_\phi\in C^2(\Omega)\cap C^{0, 1}(\overline\Omega)$ solving equation \eqref{eqn:Gauss curvature in terms of u} with $u_\phi\equiv \phi$ on $\partial \Omega$ (see \cite[Theorem 1.1]{TrudingerUrbas83}). Then one can consider the map
\begin{align*}
  \phi \mapsto \mathcal{G}u_\phi(\Omega) \subset \mathbb{S}^n_-.
\end{align*}
This map can be viewed as a ``Dirichlet-to-second boundary value map'' and contains information on the relationship between the Dirichlet problem and the optimal transport version of the prescribed Gauss curvature problem. It would be of interest to analyze various properties of this map: for example a natural question is the following: given a geodesically convex set $\Omega^* \subset \mathbb{S}^n_-$, what can we say about the set of functions $\phi$ for which $\Omega_\phi^* = \mathcal{G}u_\phi(\Omega)$? 

\section{Basics of Optimal transport}\label{section: OT}
First we recall the optimal transport problem. We will write $\mathcal{P}(X)$ to denote all Borel probability measures on a topological space $X$.
\begin{definition}
  Let $\Omega$, $\Omega^*$ be Polish spaces (separable, completely metrizable spaces), $c: \Omega\times \Omega^*\to \R\cup \{+\infty\}$ Borel measurable, and $\mu\in \mathcal{P}(\Omega)$, $\nu\in \mathcal{P}(\Omega^*)$. A solution of the \emph{Monge problem} taking $\mu$ to $\nu$ is a Borel map $T: \Omega\to \Omega^*$ defined $\mu$-a.e. such that
  \begin{align}\label{eqn: M-OT}\tag{M-OT}
      \int_\Omega c(x, T(x))d\mu(x)=\min_{S_\#\mu=\nu} \int_\Omega c(x, S(x))d\mu(x),
  \end{align}
  where the \emph{pushforward measure} under $S$ is defined by $S_\#\mu(E):=\mu(S^{-1}(E))$ for any Borel $E\subset \Omega^*$.
  
  Let $$\Pi(\mu, \nu)=\{\pi\in \mathcal{P}(\Omega\times \Omega^*)\mid (\proj_\Omega)_\#\pi=\mu, (\proj_{\Omega^*})_\#\pi=\nu \}.$$ A solution of the \emph{Kantorovich problem} between $\mu$ and $\nu$ is a measure $\pi\in \Pi(\mu, \nu)$ satisfying 
    \begin{align}\label{eqn: K-OT}\tag{K-OT}
      \int_{\Omega\times \Omega^*} c(x, \ybar)d\pi(x, \ybar)=\min_{\tilde{\pi}\in\Pi(\mu,\nu)} \int_{\Omega\times \Omega^*} c(x, \ybar)d\tilde{\pi}(x, \ybar).
  \end{align}
\end{definition}
The existence and regularity theory for optimal transport is intricately related to the theory of $c$-convex functions which we recall here.
\begin{definition}
$u: \Omega\to \R\cup \{+\infty\}$ is \emph{$c$-convex} if there is a set $\mathcal{A}\subset \Omega^* \times \R$ such that 
\begin{align*}
u(x)&=\sup_{(\ybar, \lambda)\in \mathcal{A}}(-c(x,\ybar)+\lambda).    
\end{align*}
\end{definition}
We also recall the $c$-subdifferential. 
\begin{definition}
  If $u$ is $c$-convex we say a point \emph{$\ybar \in \Omega^*$ is supporting to $u$ at $x_0 \in \Omega$} if
  \begin{align*}
    u(x) \geq u(x_0)+c(x_0, \ybar)-c(x_0, \ybar) \textnormal{ for all } x \in \Omega.
  \end{align*}
  The set of points $\ybar$ supporting  to $u$ at $x$ will be denoted by $\partial_cu(x)$, and is called the \emph{$c$-subdifferential of $u$ at $x$}. This defines a map from $\Omega$ into Borel subsets of $\Omega^*$. 
\end{definition}
\begin{remark}\label{rem: convex case}
  In the case when $\Omega\subset \R^n$ and $\Omega^*=\R^n$ with $c(x, \ybar)=-\inner{x}{\ybar}$, a $c$-convex function is a convex (lower semicontinuous) function and the $c$-subdifferential is the usual subdifferential. We will simply denote this case by $\subdiff{u}{x}$.
\end{remark}
Finally, we recall the following standard condition for existence theory.
\begin{definition}
Suppose $\Omega$ and $\Omega^*$ are subsets of smooth manifolds. A smooth cost function $c$ is said to be \emph{bi-Twisted} if for any $x_0\in \Omega$ and $\ybar_0\in\Omega^*$, the maps
\begin{align*}
    \Omega^*\ni\ybar&\mapsto -D_x c(x_0, \ybar),\\
    \Omega\ni x&\mapsto -D_{\ybar}c(x, \ybar_0).
\end{align*}
are $C^1$ diffeomorphisms. The inverses of these maps are called the \emph{$c$- and $c^*$-exponential maps}, denoted by $\cExp{x_0}{\cdot}$ and $\cstarExp{\ybar_0}{\cdot}$ respectively.
\end{definition}
Then the following existence theorem can be deduced, for example from \cite[Theorem 5.10 (iii) and Theorem 5.30]{Vil09}. This existence result originates in work of Brenier for the cost given by distance squared in Euclidean space, McCann for the case of general Riemannian manifolds, Gangbo and McCann for strictly convex or concave functions of the distance on Euclidean space, and Ma, Trudinger, and Wang for more general bi-Twisted cost functions (see \cite{Bre91, McC01, GM96, MTW05}). We note that the referenced theorems apply to domains that are open subsets of smooth manifolds, as a Polish space is a separable, completely \emph{metrizable} metric space (and thus does not have to be complete itself).
\begin{theorem}\label{theorem: OT}
Suppose $\Omega$ and $\Omega^*$ are open subsets of smooth Riemannian manifolds, $c$ is smooth and bi-Twisted, and $\mu$ and $\nu$ are Borel probability measures on $\Omega$ and $\Omega^*$ respectively, where $\mu$ is absolutely continuous with respect to the volume measure on $\Omega$. 
Additionally, suppose there exist continuous, real valued functions $(a, b)\in L^1(\mu)\times L^1(\nu)$ such that $\norm{c(x, \ybar)}\leq a(x)+b(\ybar)$ for all $(x, \ybar)\in \Omega\times \Omega^*$. 
If the minimum value in \eqref{eqn: K-OT} is finite, then there exists a $c$-convex function $u:\Omega \to \mathbb{R}\cup \{+\infty\}$ and the map $T_u: \Omega\to \Omega^*$, $T_u(x):=\cExp{x}{Du(x)}$ (defined for $\mu$-a.e. $x$) is a minimizer in \eqref{eqn: M-OT}.

We will refer to such a function $u$ as a \emph{Brenier solution} of \eqref{eqn: M-OT}.
\end{theorem}
\begin{remark}
It is well known that if $\mu=f \dVol[\Omega]$ and $\nu=g\dVol[\Omega^*]$, under relatively mild conditions on the cost function we have for almost every $x$,
  \begin{align}\label{eqn: jacobian equation}
      \det(DT_u(x))=\frac{f(x)}{g(T_u(x))},
  \end{align}
(for an appropriate weak interpretation of $\det(DT_u(x))$). The form of this pointwise a.e. equation will motivate our choice of cost function in the next section.
\end{remark}

\section{Prescribed Gauss curvature as optimal transport with hemisphere target}\label{section:Gauss Curvature as OT}
\subsection{Geometric motivation}\label{section:geometric motivation}
Consider a convex domain $\Omega\subset \R^n$, and a function $u: \Omega\to \R$ convex, smooth, and let 
\begin{align*}
   {\Sigma_u}:=\graph(u)=\{(x, u(x))\in \R^{n+1}\mid x\in \Omega\}.
\end{align*}
Define also the map $\Phi_u(x):=(x, u(x))$, which sends $\Omega$ to ${\Sigma_u}$. In more classical terminology $\Phi_u$ is a parametrization of ${\Sigma_u}$.

Now let $N_{\Sigma_u}: {\Sigma_u} \to \S^n$ be the Gauss map (i.e., $N_{\Sigma_u}(\zbar)$ is the outward normal vector to ${\Sigma_u}$ at $\zbar\in {\Sigma_u}$). Then
\begin{align*}
    \det D(N_{\Sigma_u}\circ \Phi_u)(x)&=(\det DN_{\Sigma_u}(\Phi_u(x)))(\det D\Phi_u(x))=G_{\Sigma_u}(\Phi_u(x))(\det D\Phi_u(x)),
\end{align*}
where $G_{\Sigma_u}(\zbar)$ is the Gauss curvature of ${\Sigma_u}$ at $\zbar\in {\Sigma_u}$. Since the Gauss curvature is described by a prescribed change of volume induced by the Gauss map, whose range lies in $\S^n_-$, it is natural to see if this situation can be described by an optimal transport problem where the target measure $\nu$ is supported in $\S^n_-$. We now make some heuristic calculations and find there turns out to be an appropriate choice for $\nu$.

We first recall a classical calculation, which is provided here for completeness.
\begin{proposition}\label{prop: Jacobian of Phi_u}
  Let $x\in \Omega$. The determinant of $(D\Phi_u)_x:T_x\Omega\to T_{\Phi_u(x)}{\Sigma_u}$ is given by
  \begin{align*}
    \det((D\Phi_u)_x) = \sqrt{1+|\nabla u(x)|^2}.      
  \end{align*}  
\end{proposition}

\begin{proof}
Fix $x\in \Omega$ and suppose $\nabla u(x)\neq 0$. Let $v_1, \ldots, v_{n-1}\in \R^n$ be such that 
\begin{align*}
  \{v_1,\ldots, v_{n-1}, \nabla u(x)\}
\end{align*}
is an orthogonal set with $\norm{v_i}=1$ for $i=1,\ldots, n-1$ and define $v_n:=\frac{\nabla u(x)}{\norm{\nabla u(x)}}$. If 
\begin{align*}
    w_i:&=\begin{cases}
    (v_i, 0)\in \R^{n+1},&i=1,\ldots, n-1,\\
    \frac{(\nabla u(x), \norm{\nabla u(x)}^2)}{\norm{\nabla u(x)}\sqrt{1+\norm{\nabla u(x)}}},& i=n
    \end{cases}
\end{align*} then $(\nabla u(x), -1)\perp w_i$ for all $i=1,\ldots, n$. Thus the collection $\{w_1, \ldots, w_n\}$ is an orthonormal basis of $T_{\Phi_u(x)}{\Sigma_u}$. Then we calculate,
\begin{align*}
    \left.\frac{d}{dt} \Phi_u(x+tv_i)\right \vert_{t=0}&=\left.\frac{d}{dt} (x+tv_i, u(x+tv_i))\right \vert_{t=0}\\
    &=(v_i, \inner{\nabla u(x)}{v_i})\\
    &=\begin{cases}
    w_i, &i=1,\ldots, n-1,\\
    (\frac{\nabla u(x)}{\norm{\nabla u(x)}}, \norm{\nabla u(x)})=w_n\sqrt{1+\norm{\nabla u(x)}^2},&i=n.
    \end{cases}
\end{align*}
Then 

\begin{align*}\det D\Phi_u(x)=\det \begin{pmatrix}\Id_{n-1}&0\\
0&\sqrt{1+\norm{\nabla u(x)}^2}\end{pmatrix}=\sqrt{1+\norm{\nabla u(x)}^2}.
\end{align*}

If $\nabla u(x)=0$, let $\{v_1, \ldots v_n\}$ be an arbitrary orthonormal basis of $\R^n$, clearly by defining $w_i=(v_i, 0)$ for all $i=1,\ldots, n$ we obtain an orthonormal basis for $T_{\Phi_u(x)}{\Sigma_u}$ in this case as well. Thus calculating gives $\det D\Phi_u(x)=\det \Id_n=1=\sqrt{1+\norm{\nabla u(x)}^2}$.

\end{proof}

Now suppose $K: \Omega \to \R_+$ is a smooth, nonnegative function and suppose we wish to prescribe the Gauss curvature of ${\Sigma_u}$ as $K$, i.e., we want $G_{\Sigma_u}(\Phi_u(x))=K(x)$ at every $x\in \Omega$. Since $N(\Phi_u(x))=\frac{(\nabla u(x), -1)}{\sqrt{1+\norm{\nabla u(x)}^2}}$, using Proposition \ref{prop: Jacobian of Phi_u} we see that it would suffice to have
\begin{align*}
   \det(D(N\circ \Phi_u)(x))=\frac{K(x)}{-N(\Phi_u(x))_{n+1}} 
\end{align*}
as this would imply
\begin{align*}
    G_{\Sigma_u}(\Phi_u(x))&=\frac{\det(D(N\circ \Phi_u)(x))}{\sqrt{1+\norm{\nabla u(x)}^2}}\\
    &=\frac{K(x)}{-N(\Phi_u(x))_{n+1}\sqrt{1+\norm{\nabla u(x)}^2}}\\
    &=K(x).
\end{align*}
Comparing this expression with the equation in \eqref{eqn: jacobian equation}, we find that if is possible to choose a cost function $c$ in such a way that $c$-convex functions are convex and $T_u=\cExp{\cdot}{Du(\cdot)}=N\circ \Phi_u$ at points where $u$ is differentiable, we can obtain the above relation with the choice of source measure on $\Omega$ which has density $K$ and a target measure on a subdomain $\Omega^*$ of $\S^n_-$ with density $-y_{n+1}$ with respect to $\dVol[\S^n]$, the canonical volume form on $\S^n$. A key fact is that $\det(D\Phi_u(x))$ is a function of $\nabla u(x)$, but not the value $u(x)$, which makes the selection of a target measure $\nu$ possible.

In order to be able to find such a subdomain, it is necessary that
\begin{align*}
    \int_\Omega K&\leq -\int_{\S^n\cap \{y_{n+1}<0\}}y_{n+1}\dVol[\S^n](y)=\int_{B^{\R^n}_1(0)}\frac{\sqrt{1-\norm{x}^2}}{\sqrt{1-\norm{x}^2}}dx=\omega_n,
\end{align*}
where $\omega_n$ is the Lebesgue measure of the unit ball in $\R^n$. Note this condition exactly matches the classical (sharp) necessary condition for solvability of the prescribed Gauss curvature problem.

Guided by this geometric intuition, we will introduce an appropriate cost function defined on a subdomain of $\R^n\times \S^n_-$.

\subsection{A cost function with hemisphere target}

We now analyze the prescribed Gauss curvature problem in terms of the optimal transport problem with target measure as described in the previous subsection, this will lead us to the proof of Theorem \ref{theorem:existence and uniqueness}. We will use $x$ to denote points in $\mathbb{R}^n$  and $\ybar$ to denote points in $\mathbb{S}^n$. Moreover, we shall write $\ybar=(y, y_{n+1})$ where $y\in \R^n$, $y_{n+1} \in \R$,  and write $\S^n_-:=\{\ybar\in \S^n\mid y_{n+1}<0\}$ for the open lower hemisphere. For this section, we will be interested in the optimal transport problem with $\Omega\subset \R^n$ and $\Omega^*\subset \S^n_-$, along with the cost function $c: \Omega\times \Omega^*\to \R$ defined by 
\begin{align*}
    c(x,\ybar)=\frac{\inner{x}{y}}{y_{n+1}}.
\end{align*}

It is evident that for our choice of cost function $c$, $u$ is $c$-convex if and only if it is convex and lower semicontinuous. Moreover it is also clear that if $\Omega^*$ is compactly contained in $\S^n_-$, then a $c$-convex function is uniformly Lipschitz on $\Omega$.

First we show the bi-Twistedness condition.
 \begin{proposition}\label{proposition:c exponential at x}
   The cost function $c$ is bi-Twisted, with 
  \begin{align*}
    \cExp{x}{p} = \left ( \frac{p}{\sqrt{1+|p|^2}}, - \frac{1}{\sqrt{1+|p|^2}} \right )       
  \end{align*}  
  for any $x \in \Omega$ and $p \in \mathbb{R}^n$.
 \end{proposition}
 \begin{proof}
 Clearly (after identifying the tangent and cotangent spaces of $\R^n$ with $\R^n$)
 \begin{align*}
    -D_x c(x,\ybar) = -\frac{y}{y_{n+1}},
  \end{align*}
  which is $C^1$ on $\S^n_-$ by inspection, for a fixed $x\in \Omega$.
  
Fix $p\in \R^n$ and let $\ybar = (y,y_{n+1})$ be such that $-D_x c(x,\ybar) = p$, then since $\ybar\in \S^n_-$,
  \begin{align*}
    p = -\frac{y}{y_{n+1}}=\frac{y}{\sqrt{1-\norm{y}^2}},  
  \end{align*}
  while $\norm{p}^2 = \frac{\norm{y}^2}{1-\norm{y}^2}$ yields
  \begin{align*}
y_{n+1} = -\frac{1}{\sqrt{1+\norm{p}^2}}.
  \end{align*}
  Combined, this gives the claimed formula for $\cExp{x}{p}$ and shows $-D_xc(x, \cdot)$ is invertible; by inspection $\cExp{x}{\cdot}$ is also $C^1$.
  
  On the other hand, (viewing $T_{\ybar}\S^n$ as embedded in $\R^{n+1}$)
   \begin{align}\label{eqn: D_y}
    -\nabla_{\ybar}c(x,\ybar) = \proj_{T_{\ybar}\S^n}\left(-\frac{x}{y_{n+1}}, \frac{\langle x,y\rangle}{y_{n+1}^2}\right)=\left(-\frac{x}{y_{n+1}}, \frac{\langle x,y\rangle}{y_{n+1}^2}\right),
  \end{align}
  which is again $C^1$ by inspection on $\Omega$ for a fixed $\ybar$. If $-\nabla_{\ybar}c(x,\ybar)=\qbar$ for some $\qbar=(q, q_{n+1})\in T_{\ybar}\S^n\subset R^{n+1}$, then we immediately have $x=y_{n+1} q$.  This shows both the invertibility of $-\nabla_{\ybar}c(\cdot, \ybar)$ on $\Omega$ and $C^1$-ness of its inverse, hence also of $-D_{\ybar}c(\cdot, \ybar)$, thus we have shown bi-Twistedness of $c$.
 \end{proof}

Next we verify what was suggested by the computation in Section \ref{section:geometric motivation}. Namely, we are going to show the $c$-exponential map has the desired geometric properties in order to relate our optimal transport problem with the Gauss map. In the following, for a function $u: \Omega \to \R\cup \{+\infty\}$, we write $\epi(u):=\{(x, x_{n+1})\in \Omega\times \R\mid x_{n+1}\geq u(x)\}$ for the \emph{epigraph} of $u$, and $N_{(x, u(x))}(\epi(u))=\{\vbar\in \S^n\mid \inner{\vbar}{\zbar-(x, u(x))}\leq 0,\ \forall\; \zbar\in \epi(u)\}$ is the set of outward unit normals to $\epi(u)$ at $(x, u(x))$. See also Remark \ref{rem: convex case}.

\begin{proposition}\label{prop: gauss map c-exp}
  Let $u: \Omega \to \R\cup \{+\infty\}$ be $c$-convex. Then for any $x\in \Omega$ and $p\in \subdiff{u}{x}$, 
  \begin{align*}
      \cExp{x}{p}\in \normal{(x, u(x))}{\epi(u)}.
  \end{align*}
\end{proposition}
\begin{proof}
Let $x\in \Omega$, $p\in\subdiff{u}{x}$. Then by Proposition \ref{proposition:c exponential at x}
\begin{align*}
    \cExp{x}{p}=(\frac{p}{\sqrt{1+\norm{p}^2}}, -\frac{1}{\sqrt{1+\norm{p}^2}}).
\end{align*}
Now if $\zbar=(z, z_{n+1})\in \epi(u)$, we must have $z\in \Omega$ and $u(z)\leq z_{n+1}$. Thus
\begin{align*}
    \inner{\zbar-(x, u(x))}{\cExp{x}{p}}&=\frac{\inner{z-x}{p}-(z_{n+1}-u(x))}{\sqrt{1+\norm{p}^2}}\\
    &\leq \frac{u(x)+\inner{z-x}{p}-u(z)}{\sqrt{1+\norm{p}^2}} \leq 0
\end{align*}
since $p\in \subdiff{u}{x}$. Thus $\cExp{x}{p}\in \normal{(x, u(x))}{\epi(u)}$.
\end{proof}
In particular, if $u$ is differentiable at $x\in \Omega$ we have $\subdiff{u}{x}=\{\nabla u(x)\}$ and since it is clear that $\cExp{x}{\nabla u(x)}$ is unit length, we must have that $\cExp{x}{\nabla u(x)}=N_{\Sigma_u}(\Phi_u(x))$ in the notation of Section \ref{section:geometric motivation}.

\begin{remark}\label{rem: euclidean density}
  
It turns out that the optimal transport formulation on the hemisphere we have introduced here is equivalent to the well known formulation with target domain $\R^n$ (see \cite{Urbas84}). We start by observing that $-D_{x}c(x, \cdot): \S^n_-\to \R^n$ is independent of $x\in \Omega$ and a diffeomorphism, let us denote it by $\Ybar(\cdot)$ (in particular, $\Ybar^{-1}=\cExp{x}{\cdot}$, also independent of $x$). We now calculate $\Ybar_\# (-y_{n+1}\mathds{1}_{\Omega^*})\dVol[\S^n]$ for some domain $\Omega^*\subset \S^n_-$. Fix a point $p\in \R^n$ and a standard basis vector $e_i\in \R^n$, then $v_i:=\frac{d}{dt}\Ybar^{-1}(p+te_i)\vert_{t=0}\in \R^{n+1}$ yields a vector tangent to $\S^n_-$. Calculating gives
\begin{align*}
    v_i&=\frac{d}{dt}\left(\frac{p+te_i}{\sqrt{1+\norm{p+te_i}^2}}, -\frac{1}{\sqrt{1+\norm{p+te_i}^2}}\right)\vert_{t=0}\\
    &=\frac{1}{(1+\norm{p}^2)^{\frac{3}{2}}}\left(e_i(1+\norm{p}^2)-p\inner{p}{e_i}, \inner{p}{e_i}\right)\\
    &=\frac{1}{(1+\norm{p}^2)^{\frac{3}{2}}}\left(e_i(1+\norm{p}^2)-p_ip, p_i\right).
\end{align*}
Thus viewing $\Ybar$ as a (global) coordinate chart on $\S^n_-$ we can write the coefficients of the Riemannian metric on $\S^n_-$ in this chart as
\begin{align*}
    g_{ij}&=\inner{v_i}{v_j}_{\R^{n+1}}=\frac{\delta_{ij}(1+\norm{p}^2)^2-2p_ip_j(1+\norm{p}^2)+p_ip_j\norm{p}^2+p_ip_j}{(1+\norm{p}^2)^3}\\
    &=\frac{\delta_{ij}(1+\norm{p}^2)-p_ip_j}{(1+\norm{p}^2)^2}.
\end{align*}
As a projection matrix, we see the eigenvalues of $g_{ij}$ are $(1+\norm{p}^2)^{-2}$ with multiplicity $1$, and $(1+\norm{p}^2)^{-1}$ with multiplicity $n-1$. Hence we calculate the canonical volume form on $\S^n_-$ in this chart as
\begin{align*}
    \sqrt{\det g_{ij}}dp&=\sqrt{(1+\norm{p}^2)^{-n-1})}=\frac{1}{(1+\norm{p}^2)^{\frac{n+1}{2}}}.
\end{align*}
Thus we find that
\begin{align*}
    \Ybar_\#(-y_{n+1}\mathds{1}_{\Omega^*})\dVol[\S^n]=\frac{-y_{n+1}\circ \Ybar^{-1}(p)}{(1+\norm{p}^2)^{\frac{n+1}{2}}}dp=\frac{1}{(1+\norm{p}^2)^{\frac{n+2}{2}}}dp
\end{align*}
on $\R^n$. Equivalently, this formula provides the expression for the form $-y_{n+1}\dVol[\Omega^*]$ in the coordinate chart defined by $\bar Y$. 

Thus if $T_u(x)=\cExp{x}{Du(x)}$ where $u$ is a Brenier solution as in the statement of Theorem \ref{theorem:existence and uniqueness}, we see that $Du=\Ybar\circ T_u$ satisfies 
\begin{align*}
    Du_\# Kdx=\Ybar_\#(T_u)_\#Kdx=\Ybar_\#(-y_{n+1}\mathds{1}_{\Omega^*})\dVol[\S^n]=\frac{\mathds{1}_{\Ybar^{-1}(\Omega^*)}}{(1+\norm{p}^2)^{\frac{n+2}{2}}}dp.
\end{align*}
In particular, if $\Ybar^{-1}(\Omega^*)$ is convex (i.e., if $\Omega^*$ is $c^*$-convex in the language of Definition \ref{def: c-convex domains} below), we can apply \cite[Lemmas 1 and 2]{Caf92} to see that $u$ is actually a ``generalized solution'' as defined by Urbas in \cite[Section 2]{Urbas84}.

\end{remark}

We are now ready to give the proof of our existence result Theorem \ref{theorem:existence and uniqueness}. For the first existence portion, we could reduce the problem to the usual prescribed Gauss curvature problem with target measure on $\R^n$, but here we elect to show existence solely in the framework of optimal transport with hemisphere target, to illustrate the simplicity of the approach.

\begin{proof}[Proof of Theorem \ref{theorem:existence and uniqueness}]

We consider the Kantorovich problem with cost $c(x,\ybar)$ and
\begin{align*}
  \mu = K \mathds{1}_{\Omega}\;dx,\;\nu = -y_{n+1}\mathds{1}_{\Omega^*}\dVol[\mathbb{S}^n].
\end{align*}
First we note $\norm{c(x, \ybar)}\leq \diam(\Omega)\frac{\norm{y}}{y_{n+1}}$ for any $(x, \ybar)\in \Omega\times \Omega^*$, where $\frac{\norm{y}}{y_{n+1}}\in L^1(-y_{n+1}\mathds{1}_{\Omega^*}\dVol[\S^n_-])$, and is also continuous on $\S^n_-$. Additionally, $\mu\otimes\nu\in \Pi(\mu, \nu)$ and 
\begin{align*}
    \int_{\Omega\times\Omega^*}c(x, \ybar)d(\mu\otimes\nu)(x, \ybar)&\leq\diam(\Omega)\int_\Omega K \int_{\Omega^*}\norm{y}\dVol[\S^n](\ybar)<\infty,
\end{align*}
showing that the value in \eqref{eqn: K-OT} will also be finite. Thus we can apply Theorem \ref{theorem: OT} to obtain existence of a Brenier solution $u$ from $\mu$ to $\nu$.

Since $Du_\#Kdx=\frac{dp}{(1+\norm{p}^2)^{\frac{n+2}{2}}}$ by Remark \ref{rem: euclidean density}, we may apply \cite[Proposition 4.2 and Theorem 4.4]{McCann97} to see $u$ solves equation \eqref{eqn:Gauss curvature in terms of u} almost everywhere in $\Omega$, and everywhere if $u$ happens to be $C^2$.

Finally, we can follow the proof of \cite[Theorem 10.38]{Vil09} to show the claim \eqref{eqn: gauss map image}. Indeed, it is sufficient to show that for any $x\in \overline{\Omega}\cap \dom(Du))$ and $\epsilon>0$, there exists $\delta>0$ such that $T_u(B_\delta(x)\cap \dom(Du))\subset B_\epsilon(T_u(x))$. If $x\in \dom(Du)$, then $T_u(x)\not\in \partial\S^n_-$, hence the claim follows as convexity of $u$ implies $Du$ is continuous on $\dom(Du)$ (see \cite[Theorem 25.5]{Roc70}).
\end{proof}

\subsection{Regularity and $c^*$-convexity of the target} 
We now turn to the regularity result, Theorem \ref{theorem:subcritical case regularity}.

First we recall that a key condition in the regularity theory of optimal transport is a convexity condition on the support of the target measure $\nu$.
\begin{definition}\label{def: c-convex domains}
  We say $\Omega^*$ is \emph{$c^*$-convex with respect to $\Omega$} if for any $x\in \Omega$, the set $-D_xc(x, \Omega^*)$ is a convex subset of $T_x\Omega$. Similarly, $\Omega$ is \emph{$c$-convex with respect to $\Omega^*$} if for any $\ybar\in \Omega^*$, the set $-D_{\ybar}c( \Omega, \ybar)$ is a convex subset of $T_{\ybar}\S^n_-$.
\end{definition}
It turns out that $c^*$ and $c$-convex sets have a geometrically simple description for our choice of $c$.
\begin{proposition}\label{prop: c-convex is geodesically convex}

\begin{enumerate}
    \item A subset $\Omega^*\subset \S^n_-$ is $c^*$-convex with respect to any $\Omega\subset \R^n$ if and only if it is a geodesically convex subset of $\S^n_-$.
    \item A subset $\Omega\subset \R^n$ is $c$-convex with respect to any $\Omega^*\subset \S^n_-$ if and only if it is a convex subset of $\R^n$.
\end{enumerate}

\end{proposition}
\begin{proof}
To show claim (1), it suffices to prove that the image of any line segment in $\R^n$ under $\cExp{x}{\cdot}$ for a point $x\in \R^n$ lies on a geodesic in $\S^n$. To this end, let $\ybar_0$, $\ybar_1\in \S^n_-$, and suppose $\Nbar=(N, N_{n+1})\in \S^n$ is any unit vector orthogonal to the span of $\ybar_0$ and $\ybar_1$, thus
\begin{align}\label{eqn: orthogonal}
    \inner{N}{y_i}=N_{n+1}\sqrt{1-\norm{y_i}^2},\quad i=0,1.
\end{align}
 Then let $p_i:=-D_xc(x, \ybar_i)=\frac{y_i}{\sqrt{1-\norm{y_i}^2}}$ for $i=0$, $1$, and for $t\in [0, 1]$ write $p_t:=(1-t)p_0+tp_1$. We can then calculate,
 \begin{align*}
     \inner{\Nbar}{\cExp{x}{p_t}}&=\inner{(N, N_{n+1})}{(\frac{p_t}{\sqrt{1+\norm{p_t}^2}}, -\frac{1}{\sqrt{1+\norm{p_t}^2}})}=\frac{1}{\sqrt{1+\norm{p_t}^2}}(\inner{N}{p_t}-N_{n+1})\\
     &=\frac{1}{\sqrt{1+\norm{p_t}^2}}((1-t)\inner{N}{\frac{y_0}{\sqrt{1-\norm{y_0}^2}}}+t\inner{N}{\frac{y_1}{\sqrt{1-\norm{y_1}^2}}}-N_{n+1})=0
 \end{align*}
 by \eqref{eqn: orthogonal}. Since $\Nbar$ is an arbitrary normal to $\spn(\ybar_0, \ybar_1)$, this shows $\cExp{x}{p_t}$ lies on a great circle of $\S^n$ containing $\ybar_0$ and $\ybar_1$, hence on a geodesic segment connecting the two points.
 
 On the other hand, by \eqref{eqn: D_y} we see $-D_{\ybar}c(\cdot, \ybar)$ is a linear map for fixed $\ybar$, the second claim is clear.
\end{proof}

The other key structural condition in regularity theory of optimal transport is known as the (weak) Ma-Trudinger-Wang \eqref{MTW} condition.
\begin{definition}
  A cost function $c$ satisfies the \emph{weak Ma-Trudinger-Wang condition} if for any $(x, \ybar)\in \Omega\times \Omega^*$ and $V\in T_x\Omega$, $\eta\in T^*_x\Omega$ such that $\eta(V)=0$,
  \begin{align}\label{MTW}\tag{MTW}
      -(c_{ij, rs}-c_{ij, t}c^{t, q}c_{q, rs})c^{r, k}c^{s, l}(x, \ybar)V^iV^j\eta_k\eta_l\geq 0.
  \end{align}
  Here, we fix arbitrary local coordinates on $\Omega$ and $\Omega^*$ near $(x, \ybar)$, subscript indices before a comma indicate coordinate derivatives in $\Omega$ and indices after a comma indicate coordinate derivatives in $\Omega^*$. Moreover, $c^{i, j}$ are the entries of the matrix inverse of $c_{i, j}$.
\end{definition}
Of course since our cost function is linear in $x$, it trivially satisfies the weak \ref{MTW} condition.

\begin{proof}[Proof of Theorem \ref{theorem:subcritical case regularity}]
Suppose first that $\Omega^*$ is a geodesically convex, compact subset of $\S^n_-$. By Proposition \ref{prop: c-convex is geodesically convex}, $\Omega^*$ is $c^*$-convex with respect to $\Omega$, while $\Omega$ is (compactly) contained in some large ball, which is $c$-convex with respect to $\Omega^*$.  Since $-y_{n+1}$ is bounded away from zero and infinity on $\Omega^*$, we may apply the main result of \cite{GK14} to see that $u\in C^{1, \alpha}_{loc}(\Omega)$.

Now suppose $\overline{\Omega^*}\cap \partial \S^n_-\neq \emptyset$. In this case, (using Remark \ref{rem: euclidean density}) we may view a Brenier solution $u$ as a convex solution to an optimal transport problem with absolutely continuous target measure in $\R^n$, by Proposition \ref{prop: c-convex is geodesically convex} the support of this target measure is convex. Since $\frac{1}{(1+\norm{p}^2)^{\frac{n+2}{2}}}$ is bounded away from zero and infinity on any compact set, and $\norm{\partial \Omega}=0$, as $\Omega$ is uniformly convex, we may apply \cite[Theorem 2.1]{CorderoFigalli19} under condition (2-d) there to finish the proof.

\end{proof}

\section{Proof of Theorem \ref{theorem:critical case boundary behavior}}\label{section:Urbas revisited}

Finally, we can use an argument based on that of \cite[Theorem 5.1]{FKM13} to show Theorem \ref{theorem:critical case boundary behavior}. The argument uses monotonicity properties of the transport map together with the Jacobian equation to show that interior (boundary) points of $\Omega$ must map to interior (boundary) points of $\Omega^*$. Here we will consider the problem with target domain in $\R^n$, as in Remark \ref{rem: euclidean density}. Also, in this section we will use $B^n$ and $B^{n-1}$ to denote open balls in $\R^n$ and $\R^{n-1}$ respectively.

\begin{remark}\label{rem: Omega boundary}
  We make a quick observation on the boundary of a convex, compact domain. Since $\Omega$ is convex, its boundary is locally Lipschitz. Then by a compactness argument, we see there exist constants $\rho\in (0, 1)$, $L>0$, and $C_1>1$ depending only on $\partial \Omega$ with the following property: for each point $z\in \partial \Omega$, after a rotation and translation of coordinates, there is an $L$-Lipschitz function $\phi: B^{n-1}_{\rho}(0)\to (-C_1\rho, C_1\rho)$ such that
\begin{align*}
    \mathcal{C}_\rho\cap \Omega&=\{(y', y_n)\in \mathcal{C}_\rho\mid y_n>\phi(y')\},\\
    z&\in \mathcal{C}_{\rho/2},
\end{align*}
where
\begin{align*}
    \mathcal{C}_\rho:&= B^{n-1}_{\rho}(0)\times (-2C_1\rho, 2C_1\rho).
\end{align*}
\end{remark}

\begin{definition}\label{def: legendre transform}
  If $u:\Omega\to \R\cup\{+\infty\}$ is not identically $+\infty$, its \emph{Legendre transform} $u^*: \R^n\to\R\cup\{+\infty\}$ is defined by
  \begin{align*}
      u^*(p):=\sup_{x\in \Omega}(\inner{x}{p}-u(x)).
  \end{align*}
\end{definition}
Let $p_0\in \R^n$ and suppose $x_0\in \subdiff{u^*}{p_0}$, or equivalently $p_0\in \subdiff{u}{x_0}$ (see \cite[Theorem 23.5]{Roc70}). Let us also write $d_0:=d(x_0, \partial \Omega)$ for brevity, and suppose $v_0$ is a unit vector such that the point 
\begin{align}\label{eqn:x boundary}
  x_\partial:=x_0+d_0v_0,    
\end{align}
is a point in $\partial \Omega$ which achieves the closest distance $d_0$ to $x_0$.

Now for a fixed $\theta\in (0, 1)$, we define the set 
\begin{align*}
    E^*_{\theta}:=\{p\in \R^n\mid \norm{\frac{p-p_0}{\norm{p-p_0}}-v_0}\leq \theta,\ \norm{p-p_0}\leq 1\}.
\end{align*}
We will actually make the choice
\begin{align}
    \theta&=\sqrt{\frac{d_0}{6R_0}}<\frac{1}{\sqrt{6}},\label{eqn: theta condition}
\end{align}
recall that $R_0>0$ is an upper bound on the radii of all enclosing balls for $\Omega$ touching at a point on $\partial \Omega$.

If $p\in E^*_{\theta}$ and $x\in \subdiff{u}{p}$, by monotonicity of the subdifferential of convex functions (see \cite[Theorem 24.9]{Roc70}), we calculate
\begin{align*}
    \inner{x-x_0}{-v_0}&=\inner{x-x_0}{\frac{p-p_0}{\norm{p-p_0}}-v_0}-\inner{x-x_0}{\frac{p-p_0}{\norm{p-p_0}}}\\
    &\leq \inner{x-x_0}{\frac{p-p_0}{\norm{p-p_0}}-v_0}\\
    &\leq \norm{x-x_0}\norm{\frac{p-p_0}{\norm{p-p_0}}-v_0}
    \leq \theta \norm{x-x_0}.
\end{align*}
Thus if we let 
\begin{align*}
    E_{\theta}:=\{x\in \Omega\mid \inner{x-x_0}{-v_0}\leq \theta \norm{x-x_0}\},
\end{align*}
we have $\subdiff{u^*}{E^*_{\theta}}\subset E_{\theta}$ (see \cite[Fig. 1]{FKM13}).

Now we show $E_{\theta}$ is sufficiently close to the point $x_\partial$ from \eqref{eqn:x boundary}.

\begin{lemma}\label{lemma:Cone inclusion}

We have the inclusion (see Figure \ref{fig:Etheta})
\begin{align}\label{eqn: E theta inclusion}
    E_{\theta}\subset B^n_{\sqrt{(1+4R_0)d_0}}(x_\partial)\subset B^n_{\rho/4}(x_\partial).
\end{align}

\end{lemma}

\begin{proof}
 Rotate and translate coordinates on $\R^n$ so that $x_\partial$ is the origin and the vector $v_0$ is in the negative $e_n$ direction (thus $x_0=d_0e_n$). In particular, the plane $\R^{n-1}$ is tangent to $\partial \Omega$ at the origin. By the enclosing ball condition, there exists a ball $B^n$ of radius $R\leq R_0$ which contains $\Omega$ and such that $\partial B\cap \partial \Omega=\{0\}$. Then writing $x=(x', x_n)\in \R^{n-1}\times \R$, we find if $x\in E_{\theta}$, then
\begin{align}\label{eqn: x_n bound}
    x_n&\geq R-\sqrt{R^2-\norm{x'}^2}\geq \frac{\norm{x'}^2}{2R}
\end{align}
which follows from taking $t=|x'|/R$ the elementary inequality $1-\sqrt{1-t^2} \geq t^2/2$ for all $t>0$. Now suppose $x\in E_{\theta}$ with $x_n\geq d_0$. Then by the definition of $E_{\theta}$ and \eqref{eqn: x_n bound}, we find
\begin{align*}
x_n^2-2d_0x_n+d_0^2&=(x_n-d_0)^2\leq \theta^2\norm{x-d_0e_n}^2\\
&=\theta^2x_n^2+\theta^2\norm{x'}^2-2d_0\theta^2x_n+\theta^2d_0^2\\
&\leq \theta^2x_n^2+2R\theta^2x_n-2d_0\theta^2x_n+\theta^2d_0^2\\
\implies &(1-\theta^2)x_n^2-2(d_0(1-\theta^2)+R\theta^2)x_n+(1-\theta^2)d_0^2\leq 0. 
\end{align*}
Then since $1-\theta^2\geq 5/6$ and $R\theta^2\leq d_0/6$, both by \eqref{eqn: theta condition},
\begin{align*}
    d(x, \partial \Omega)&\leq x_n\leq \frac{2(d_0(1-\theta^2)+R\theta^2)+\sqrt{4(d_0(1-\theta^2)+R\theta^2)^2-4d_0^2(1-\theta^2)^2}}{2(1-\theta^2)}\\
    &\leq\frac{d_0(1-\theta^2)+\frac{d_0}{6}+\sqrt{R}\theta\sqrt{R\theta^2+2d_0-2d_0\theta^2}}{1-\theta^2}\\
    &\leq\frac{d_0(1-\theta^2)+\frac{d_0}{6}+\sqrt{\frac{d_0}{6}}\sqrt{\frac{d_0}{6}+2d_0}}{1-\theta^2}\\
    &\leq d_0+\frac{\frac{d_0}{6}+\frac{d_0\sqrt{13}}{6}}{\frac{5}{6}}\\
    &=2d_0.
\end{align*}
In the case where $x_n \leq d_0$, we simply note that $d(x,\partial \Omega) \leq x_n \leq d_0\leq 2d_0$. By \eqref{eqn: d_0 condition} this implies the inequality in \eqref{eqn: K decay condition} holds for all $x\in E_{\theta}$. At the same time recalling \eqref{eqn: x_n bound} and \eqref{eqn: d_0 condition} gives that 
\begin{align*}
    \norm{x'}^2\leq 4Rd_0,\text{ hence }\norm{x}^2\leq 4d_0^2+4Rd_0\leq (1+4R_0)d_0 ,
\end{align*}
and we also see
\begin{align*}
    E_{\theta}\subset B^n_{\sqrt{(1+4R_0)d_0}}(x_\partial)\subset B^n_{\rho/4}(x_\partial).
\end{align*}

\begin{figure}
\includegraphics[scale=0.5]{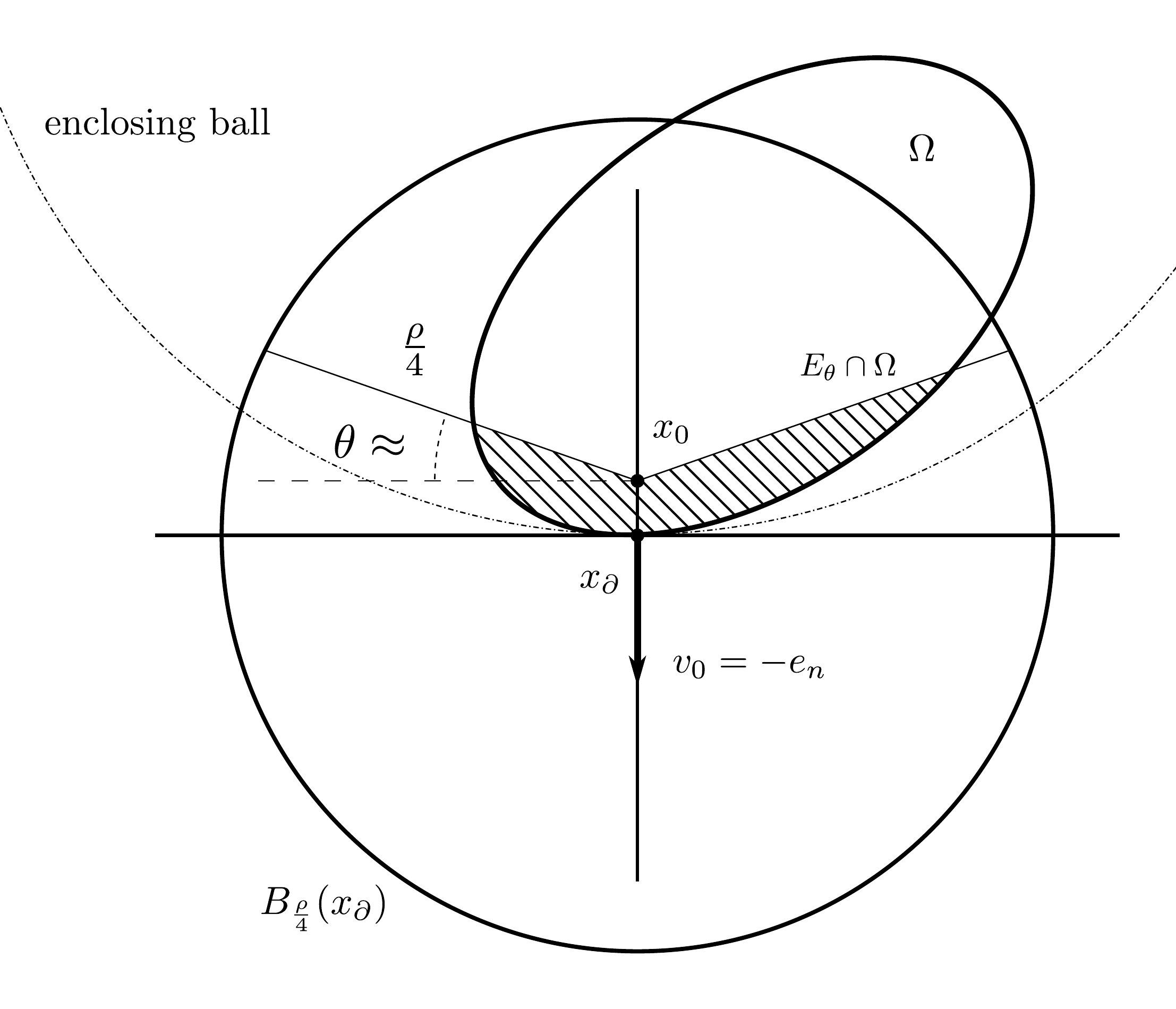}
\caption{The cone $E_\theta$ and its intersection with $\Omega$.}
\label{fig:Etheta}
\end{figure}

\end{proof}

Let us now make another change of coordinates: make the rotation and translation of coordinates as in Remark \ref{rem: Omega boundary}, and take the cylinder $\mathcal{C}_\rho$ and function $\phi$ corresponding to the point $x_\partial$ (hence $x_\partial$ may no longer be the origin). 

Let us define the cylinder
\begin{align*}
    \tilde{\mathcal{C}}:=B^{n-1}_{\sqrt{(1+4R_0)d_0}}(x'_\partial)\times (\inner{e_n}{x_\partial}-\sqrt{(1+4R_0)d_0}, \inner{e_n}{x_\partial}+\sqrt{(1+4R_0)d_0}),
\end{align*}
we will now make a detour to show an estimate on $\Hvol[n-1]{\partial \Omega_t\cap \tilde{\mathcal{C}}}$ in terms of $d_0$. 

\begin{lemma}\label{lemma:slice estimate}
  For each $t\geq 0$ define the set
  \begin{align*}
    \Omega_t = \{ x\in \Omega \mid d(x,\partial \Omega) \geq t \}.    
  \end{align*}  
  Then, we have the following estimate
  \begin{align}\label{eqn: St measure bound}
    \Hvol[n-1]{\Omega_t\cap \tilde{\mathcal{C}}}&\leq (1+4R_0)^{\frac{n-1}{2}}\omega_{n-1}(1+L)^{n-1}d_0^{\frac{n-1}{2}}.
  \end{align}
  Here, $L$ is the Lipschitz constant defined in Remark \ref{rem: Omega boundary}.
\end{lemma}

\begin{proof}
First, note that for any $t\geq 0$, the set $\Omega_t:=\{x\in \Omega\mid d(x, \partial\Omega)\geq t\}$ is convex. Indeed, suppose $x_1, x_2\in \Omega_t$, then we must have $B_{t}(x_1)\cup B_{t}(x_2)\subset \Omega$. Then, by convexity of $\Omega$, for any $\lambda\in [0, 1]$
\begin{align*}
    B_{t}((1-\lambda)x_1+\lambda x_2)&=\bigcup_{(y_1, y_2)\in B_t(x_1)\times B_t(x_2)}\{(1-\lambda)y_1+\lambda y_2\}\subset \Omega,
\end{align*}
hence $d((1-\lambda)x_1+\lambda x_2, \partial \Omega)\geq t$ as well and the claim is proved; in particular $\partial \Omega_t$ is Lipschitz. 

 If $0<t<2d_0$, we claim that $ \Omega_t\cap \tilde{\mathcal{C}}$ is given by the epigraph of an $L$-Lipschitz function over $B^{n-1}_{\sqrt{(1+4R_0)d_0}}(x'_\partial)$; note that since $x_\partial\in \mathcal{C}_{\rho/2}$, by \eqref{eqn: d_0 condition} we have $\tilde{\mathcal{C}}\subset \mathcal{C}_\rho$. To this end, fix a point $(w'_0, h_0)\in \partial \Omega_t\cap \tilde{\mathcal{C}}$, and suppose $(w'_1, h_1)$, $(w'_t, h_t)\in \tilde{\mathcal{C}}$ satisfy (see Figure \ref{fig:Omega_t})
\begin{align}
    \norm{(w'_1, h_1)-(w'_t, h_t)}&\leq t,\notag\\
    h_0+L\norm{w'_1-w'_0}&<h_1\label{eqn: estimate 1}.
\end{align}
\begin{figure}
    \centering
    \includegraphics[scale=0.7]{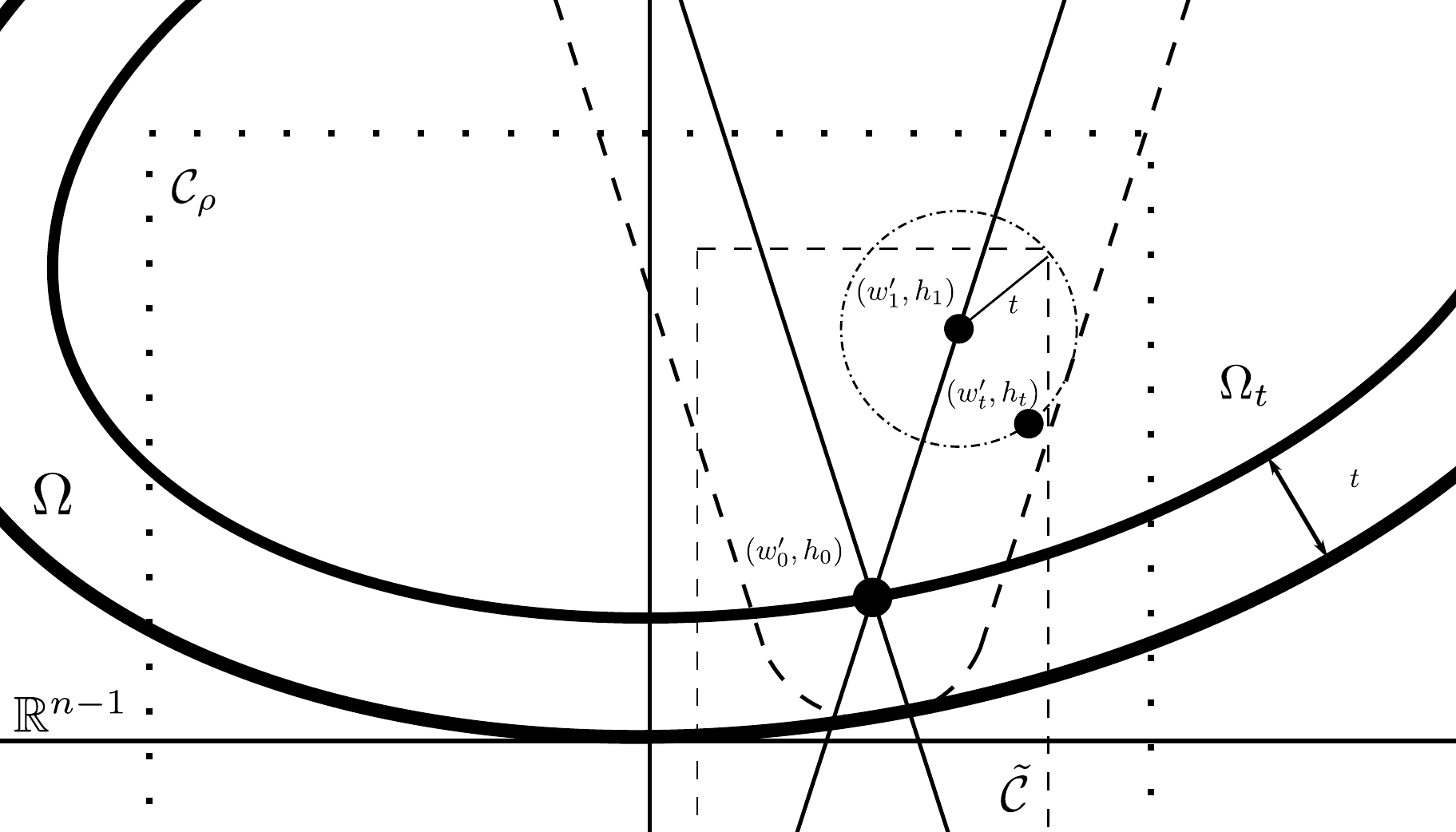}
    \caption{The sets $\Omega_t$ are locally Lipschitz for small $t$.}
    \label{fig:Omega_t}
\end{figure}

Then

\begin{align}\label{eqn: estimate 2}
t^2-\norm{w'_t-w'_1}^2\geq (h_1-h_t)^2,
\end{align}
while by \eqref{eqn: d_0 condition} and \eqref{eqn: E theta inclusion} we see $B^n_t((w'_0, h_0))\subset \Omega\cap \mathcal{C}_\rho$. We also calculate
\begin{align*}
   \norm{w'_0+w'_t-w'_1}&\leq \norm{x'_\partial}+\norm{w'_0-x'_\partial}+\norm{w'_t-w'_1}
   \leq \frac{\rho}{2}+\sqrt{(1+4R_0)d_0}+2d_0\leq \rho,
\end{align*}
hence (using \eqref{eqn: estimate 1} and \eqref{eqn: estimate 2}),
\begin{align}
    \phi(w'_0+w'_t-w'_1)+L\norm{w'_t-(w'_0+w'_t-w'_1)}&\leq h_0-\sqrt{t^2-\norm{(w'_0+w'_t-w'_1)-w'_0}^2}+L\norm{w'_1-w'_0}\notag\\
    &\leq h_1-\norm{h_1-h_t}<h_t.\label{eqn: ht bigger than cone}
\end{align}
Since $\phi$ is $L$-Lipschitz, \eqref{eqn: ht bigger than cone} implies that $(w'_t, h_t)\in\epi(\phi)\setminus \partial\Omega$. This shows that for any $(w'_0, h_0)\in \partial \Omega_t\cap \tilde{\mathcal{C}}$,
\begin{align}\label{eqn: St upper cone}
    h_0+L\norm{w'-w'_0}<h\implies (w', h)\in \Omega_t.
\end{align}
Next since $\Omega_t$ is convex, there is a supporting hyperplane to $\Omega_t$ at $x_\partial$ with an outer normal vector $N\in\R^n$, that is,
\begin{align}\label{eqn: St halfspace}
    \inner{N}{(w', h)-(w'_0, h_0)}\leq 0,\quad\forall (w', h)\in \Omega_t.
\end{align}
Suppose $(w', h)$ satisfies $h<h_0-L\norm{w'-w'_0}$. Then $2h_0-h>h_0+L\norm{w'-w'_0}=h_0+L\norm{(2w'_0-w')-w'_0)}$, hence by \eqref{eqn: St upper cone} we must have $(2w'_0-w', 2h_0-h)\in \Omega_t\setminus \partial \Omega_t$. Then using \eqref{eqn: St halfspace},
\begin{align*}
    \inner{N}{(w', h)-(w'_0, h_0)}&=-\inner{N}{(2w'_0-w', 2h_0-h)-(w'_0, h_0)}\geq 0,
\end{align*}
hence we must have $(w', h)\not\in \Omega_t$. Thus we have
\begin{align*}
    h<h_0-L\norm{w'-w'_0}\implies (w', h)\not\in \Omega_t,
\end{align*}
and by combining this with \eqref{eqn: St upper cone}, we obtain the existence of a convex, $L$-Lipschitz function $\phi_t: B^{n-1}_{\sqrt{(1+4R_0)d_0}}(x'_\partial)\to \R$ such that
\begin{align*}
    \Omega_t\cap \tilde{\mathcal{C}}=\{(x', x_n)\in \tilde{\mathcal{C}}\mid x_n\geq \phi_t(x')\}.
\end{align*}
We can then define the map $\Phi_t: B^{n-1}_{\sqrt{(1+4R_0)d_0}}(x'_\partial)\to \Omega_t\cap \tilde{\mathcal{C}}$ defined by $\Phi_t(x')=(x', \phi_t(x'))$, this is  clearly a $(1+L)$-Lipschitz map that is onto $\Omega_t\cap \tilde{\mathcal{C}}$. Thus we obtain the desired bound
\begin{align*}
    \Hvol[n-1]{\Omega_t\cap \tilde{\mathcal{C}}}&\leq\Hvol[n-1]{\Phi_t\left(B^{n-1}_{\sqrt{(1+4R_0)d_0}}(x'_\partial)\right)}\notag\\
    &\leq (1+4R_0)^{\frac{n-1}{2}}\omega_{n-1}(1+L)^{n-1}d_0^{\frac{n-1}{2}}.
\end{align*}

\end{proof}

\begin{proof}[Proof of Theorem \ref{theorem:critical case boundary behavior}]
By Remark \ref{rem: euclidean density} (recall \cite[Section 2]{Urbas84}), we see that $u$ satisfies the equation
\begin{align*}
    \int_{\partial u(E_\theta)}\frac{dp}{(1+\norm{p}^2)^{\frac{n+2}{2}}}=\int_{E_\theta}K(x)dx.
\end{align*}
Since, at the beginning of this section we showed the inclusion $E_\theta^* \subset \partial u(E_\theta)$, we then have 
\begin{align*}
  \int_{E^*_{\theta}}\frac{dp}{(1+\norm{p}^2)^{\frac{n+2}{2}}} \leq \int_{E_{\theta}}K(x)\;dx.
\end{align*}
The proof of the estimate will amount to estimating each of these two integrals. First, we have  
\begin{align}
    \int_{E^*_{\theta}}\frac{dp}{(1+\norm{p}^2)^{\frac{n+2}{2}}}
    &\geq \frac{\norm{E^*_{\theta}}}{(1+(\norm{p_0}+1)^2)^{\frac{n+2}{2}}} \geq \frac{\omega_{n-1}\theta^{n-1}}{n2^{n-1}(1+(\norm{p_0}+1)^2)^{\frac{n+2}{2}}}, \notag
\end{align}
where we made use of the straightforward estimate  $\norm{E^*_{\theta}} \geq \frac{\omega_{n-1} \theta^{n-1}}{n2^{n-1}}$. Then, in light of the definition of $\theta$, 
\begin{align*}
  \int_{E^*_{\theta}}\frac{dp}{(1+\norm{p}^2)^{\frac{n+2}{2}}} \geq\frac{\omega_{n-1}d_0^{\frac{n-1}{2}}}{n2^{n-1}6^{\frac{n-1}{2}}R_0^{\frac{n-1}{2}}(\norm{p_0}+2)^{n+2}}.\notag    
\end{align*}
Now we estimate the second integral. From the assumption on $K$,
\begin{align*}
  \int_{E_{\theta}}K & \leq C_0\int_{E_{\theta}}d(x,\partial \Omega)^{-\delta} dx.
\end{align*}
To estimate this integral we apply the co-area formula and Lemma \ref{lemma:slice estimate}, which yields
\begin{align*}
    \int_{E_{\theta}}d(x, \partial \Omega)^{-\delta} dx &\leq \int_0^{2d_0}\int_{\{x\in \Omega\mid d(x, \partial \Omega)=t,\ \norm{x'}\leq \sqrt{(1+4R_0)d_0}\}}t^{-\delta}d\Hvol[n-1]{x}dt\\
    &\leq (1+4R_0)^{\frac{n-1}{2}}\omega_{n-1}(1+L)^{n-1}d_0^{\frac{n-1}{2}}\int_0^{2d_0}t^{-\delta}dt\\
    &\leq \frac{2(1+4R_0)^{\frac{n-1}{2}}\omega_{n-1}(1+L)^{n-1}d_0^{\frac{n-1}{2}+1-\delta}}{1-\delta}.
\end{align*}
Combining these two estimates we have
\begin{align*}
\frac{\omega_{n-1} d_0^{\frac{n-1}{2}}}{n2^{n-1}6^{\frac{n-1}{2}}R_0^{\frac{n-1}{2}}(\norm{p_0}+2)^{n+2}}&\leq \frac{2(1+4R_0)^{\frac{n-1}{2}}\omega_{n-1}C_0(1+L)^{n-1}d_0^{\frac{n-1}{2}+1-\delta}}{1-\delta}.
\end{align*}

After rearranging, this inequality amounts to
\begin{align*}
  |p_0| \geq \frac{\Lambda }{d_0^{1-\delta}}-2, \textnormal{ where } \Lambda = \left ( \frac{(1-\delta)}{n2^nC_0(1+L)^{n-1}(6R_0+24R_0^2)^{\frac{n-1}{2}}} \right )^{\frac{1}{n+2}}.
\end{align*}
This is \eqref{eqn: p_0 estimate} and the theorem is proved.

\end{proof}
\bibliography{PJEreg.bib}
\bibliographystyle{alpha}
\end{document}